\documentclass[journal]{IEEEtran}  

\IEEEoverridecommandlockouts                              

\usepackage{graphicx}
\usepackage{amsmath}
\usepackage{optidef}
\usepackage{amsthm}
\usepackage{amssymb}
\usepackage{multicol}
\usepackage{cuted}
\usepackage{multirow}
\usepackage{caption, subcaption}

\newtheorem{lemma}{Lemma}
\newtheorem{proposition}{Proposition}

\newtheorem{problem}{Problem}
\usepackage{xcolor}
\usepackage{cite}
\usepackage[colorlinks=true,allcolors=steelblue]{hyperref}
\usepackage{comment}
\usepackage{algorithm}
\usepackage{algorithmic}
\usepackage{scalerel}
\usepackage{bm}
\definecolor{steelblue}{RGB}{70,130,180}

\newcommand{\ts}{\textsuperscript}

\def\bbr{\mathbb R}

\def\bbs{\mathbb S}
\def\bbp{\mathbb P}
\def\bbe{\mathbb E}

\def\calp{\mathcal P}

\newcommand{\PP}{\mathbb{P}}

\newcommand{\norm}[1]{\left\|#1 \right\|}

\renewcommand{\norm}[1]{\left\lVert {#1} \right\rVert}

\title{\LARGE \bf
Distributionally Robust Covariance Steering\\ with Optimal Risk Allocation}

\author{Venkatraman Renganathan, \and Joshua Pilipovsky, \and Panagiotis Tsiotras 
\thanks{
V. Renganathan is with the Department of Automatic Control LTH, Lund University, Lund, Sweden. J. Pilipovsky and P. Tsoitras are with the School of Aerospace Engineering, Georgia Institute of Technology, Atlanta, GA 30332-0150 USA. E-mail: venkatraman.renganathan@control.lth.se, jpilipovsky3@gatech.edu, tsiotras@gatech.edu.
}
}

\begin{document}

\maketitle
\thispagestyle{empty}
\pagestyle{empty}


\begin{abstract}
This article extends the optimal covariance steering (CS) problem for discrete time linear stochastic systems modeled using moment-based ambiguity sets. 
To hedge against the uncertainty in the state distributions while performing covariance steering, distributionally robust risk constraints are employed during the optimal allocation of the risk. Specifically, a distributionally robust iterative risk allocation (DR-IRA) formalism is used to solve the optimal risk allocation problem for the CS problem using a two-stage approach. 
The upper-stage of DR-IRA is a convex problem that optimizes the risk, while the lower-stage optimizes the controller with the new distributionally robust risk constraints. The proposed framework results in solutions that are robust against arbitrary distributions in the considered ambiguity set. Finally, we demonstrate our proposed approach using numerical simulations. Addressing the covariance steering problem through the lens of distributional robustness marks the novel contribution of this article.
\end{abstract}




\section{Introduction}

Intelligent and adaptive systems of the ``smart world'' that work under operational constraints seek to solve some instance of a constrained optimal control problem for optimizing their performance. 
Such constrained optimal control problems can now be increasingly solved efficiently using several numerical optimization techniques. 
For instance, robot path planning in uncertain environments~\cite{luders_ccrrt, luders_ccrrtstar, summers_risk, venki_risk, han_risk, ono_planning}
has gained the attention of researchers worldwide as robots are being increasingly deployed to solve many real-world problems.
Apart from realistic constraints, reliability of operation of these systems is often thwarted by the ineffective handling of system uncertainties, which can be either deterministic or stochastic. 

Control of stochastic systems can be best formulated as a problem of controlling the distribution of trajectories over time subject to constraints. Recently, the finite horizon covariance steering (CS) problem, namely, the problem of steering an initial distribution to a final distribution at a specific final time step subject to linear time varying dynamics has been explored~\cite{okamoto2018optimal}.
Specifically, the control problem in the CS problem setting involves steering the mean and the covariance to the desired terminal values. When the decision-making process relies blindly on the functional form of the process that models the stochastic uncertainty, it is known to result in potentially severe miscalculation of risk. For instance, Gaussianity assumptions made in the name of tractability in several modeling regimes are actually rarely justifiable, as the true distribution that governs the uncertain data might be non-Gaussian. Such shortcomings can be mitigated with risk-based stochastic optimization where the risk of wrong decisions can be appropriately handled to result in risk-averse decision making. One such tool is the Distributionally Robust Optimization (DRO) advocated in \cite{dr_goh, dr_wiesemann} which enables modelers to explicitly incorporate ambiguity in probability distributions into the optimization problem. 

Control of stochastic systems often involves optimizing the system's objective subject to chance constraints, where one assumes that the system uncertainties follow a known distribution 
and enforces that the system constraints hold with high probability as a function of the decision variables. 
The number of constraint violations, called the total risk budget, is usually a user-defined a priori specification and is a natural metric to assess risk.
Hence, one may consider the problem of finding a risk allocation procedure that will 
allocate the probability of violating each individual chance constraint at each time step. 
Given a number of chance constraints across a finite horizon, 
the total risk budget has to be allocated for all chance constraints across all time steps~\cite{sleiman_risk}. 
It is a common practice to consider a uniform risk allocation, i.e., allocate the same risk for all constraints and across all time steps. However, risk allocation can be optimized as in \cite{ono_iter_risk, vitus_risk_feedback} to reduce the conservatism resulting from a uniform risk allocation. 
If the probability distribution of the system uncertainties are known exactly, then non-uniform risk allocation can be performed effectively.
However, risk allocation optimization with arbitrary distribution of the system uncertainties has not yet been explored till now. 
This article addresses this shortcoming. 

\emph{Contributions:} Since authors in \cite{pilipovsky2021covariance} solved the CS problem for the Gaussian case, this article extends it with arbitrary distributions using the theory of distributional robustness (DR). 
To the best of our knowledge, this article is the first one to extend the CS problem using  distributionally robust optimization techniques for both polytopic and convex conic state constraint sets. Our main contributions in this article are as follows:

\begin{enumerate}

    \item We extend the covariance steering problem tailored between Gaussian distributions to arbitrary distributions modeled using moment based ambiguity sets.
    
    \item We enforce distributionally robust risk constraints for both polytopic and convex cone state constraint satisfaction, 
    while solving the covariance steering problem, and obtain the optimal risk allocation through a distributionally robust iterative risk allocation (DR-IRA) algorithm. 
    
    \item We demonstrate our approach using simulation examples and show the effectiveness of the proposed generalization for covariance steering problems between arbitrary distributions in moment-based ambiguity sets.
    
\end{enumerate}

Following a short summary of notation and preliminaries, the rest of the paper is organized as follows: 
The main problem statement of distributionally robust covariance steering problem with iterative risk allocation is 
presented in Section~\ref{sec_prob_formulation}. 
Then, the proposed Distributionally Robust Iterative Risk Allocation (DR-IRA) algorithm is discussed in Section~\ref{sec_dr_ira}. 
Subsequently, the proposed approach is demonstrated using simulation results in Section~\ref{sec_num_sim}. 
Finally, the paper is concluded in Section~\ref{sec_conclusion} along with directions for future research.


\section*{Notation and Preliminaries}

The set of real numbers and natural numbers are denoted by $\mathbb{R}$ and $\mathbb{N}$, respectively. The subset of natural numbers between and including $a$ and $b$ with $a < b$ is denoted by $[a,b]$. An identity matrix in dimension $n$ is denoted by $I_{n}$. For a non-zero vector $x \in \bbr^{n}$ and a matrix $P \in \bbs^{n}_{++}$ or $P \succ 0$, we denote $\left \| x \right \Vert^{2}_{P} = x^{\top} P x$, where $\bbs^{n}_{++}$ is the set of all positive definite matrices. We say $P \succ Q$ or $P \succeq Q$ if $P - Q \succ 0$ or $P - Q \succeq 0$ respectively. We denote by $\mathcal{B}(\bbr^{d})$ and $\mathcal{P}(\bbr^{d})$ the Borel $\sigma-$algebra on $\bbr^{d}$ and the space of probability measures on $(\bbr^{d}, \mathcal{B}(\bbr^{d}))$ respectively. A probability distribution with mean $\mu$ and covariance $\Sigma$ is denoted by $\mathbb{P}(\mu, \Sigma)$ and, specifically $\mathcal{N}_{d}(\mu, \Sigma)$, if the distribution is normal in $\mathbb{R}^{d}$.
The cumulative distribution function (cdf) of the normally distributed random variable is denoted by $\mathbf{\Phi}$. 
A uniform distribution over a compact set $A$ is denoted by $\mathcal{U}(A)$. Given a constant $q \in \bbr_{\geq 1}$, the set of probability measures in $\mathcal{P}(\bbr^{d})$ with finite $q-$th moment is denoted by $\mathcal{P}_{q}(\bbr^{d}) := \left\{ \mu \in \mathcal{P}(\bbr^{d}) \mid \int_{\bbr^{d}} \left \Vert x \right \|^{q} d\mu < \infty \right\}$.


\section{Problem Formulation} 
\label{sec_prob_formulation}

Consider a linear, stochastic, discrete and time-varying system as follows 
\begin{equation} \label{eqn_dt_stochastic_sys}
x_{k+1}= A_{k} x_k + B_{k} u_k + D_{k} w_k, \quad k = [0,N-1],
\end{equation}
where $x_{k} \in \bbr^n$ and $u_{k} \in \bbr^m$ is the system state and input at time $k$, respectively and $N$ denotes the total time horizon. 
Further, $A_{k} \in \mathbb{R}^{n \times n}$ is the dynamics matrix, $B_{k} \in \mathbb{R}^{n \times m}$ is the input matrix and $D_{k} \in \mathbb{R}^{n \times r}$ is the disturbance matrix. 
The process noise $w_k \in \bbr^r$ is a zero-mean random vector that is independent and identically distributed across time. We assume that the system state and the process noise is independent of each other at all time steps, meaning that $\bbe[x_{k} w_{j}] = 0$, for $0 \leq k \leq j \leq N$. 
The distribution $\mathbb{P}_{w}$ of $w_k$ is unknown but is assumed to belong to a moment-based ambiguity set of distributions, $\calp^w$ given by
\begin{equation}
\calp^w = \left\{ \bbp_{w} \mid \bbe[w_k] = 0, \bbe[w_k w_k^{\top}] = \Sigma_{w} \right\}.
\end{equation}
Note that there are infinitely many distributions present in the considered set $\calp^w$. 
For instance, both multivariate Gaussian and multivariate Laplacian distributions with zero mean and covariance $\Sigma_{w}$ belong to $\calp^w$ with the latter having heavier tail than the former. 
We assume that the system is controllable under zero process noise meaning that given any $x_{0}, x_{f} \in \mathbb{R}^{n}$, there exists a sequence of control inputs $\{u_k\}^{N-1}_{k=0}$ that steers the system state from $x_{0}$ to $x_{f}$. 
Since the considered system is stochastic, the initial condition $x_{0}$ is subject to a similar uncertainty model as the noise, with the distribution belonging to a moment-based ambiguity set, $\mathbb{P}_{x_{0}} \in \calp^{x_0}$, where 
\begin{equation} \label{eqn_ambig_x_0}
\footnotesize
    \calp^{x_0} = \left\{ \PP_{x_{0}} \mid \bbe[x_0] = \mu_0, \bbe[(x_0-\mu_0) (x_0-\mu_0)^{\top}] = \Sigma_{0} \right\}.
\end{equation}
Provided that the control law $u_k$ is selected as an affine function of state $x_k$ at any time $k$, similar moment-based ambiguity sets $\calp^{x_k}$ can be written for the distribution of states at any time $k \in [1,N]$ using the propagated mean and covariance at time $k$. 
Note that $\calp^{x_k}$ would \emph{not} be empty at all time steps $k \in [1,N]$ as a Gaussian distribution would be a guaranteed member in that set. 
This is because, Gaussianity is preserved under linear transformations defined by the dynamics. However, the initial state $x_0$ need \emph{not} be Gaussian, and thereby it would be \emph{inappropriate} to assume that the final state $x_N$ would also be Gaussian. 
Hence, the terminal state $x_{N}$ is subject to a similar uncertainty model as the $x_0$, with its distribution $\mathbb{P}_{x_{N}}$ belonging to a moment-based ambiguity set $\calp^{x_N}$, given by 
\begin{equation} \label{eqn_ambig_x_N}
\footnotesize
    \calp^{x_N} = \left\{ \PP_{x_{N}} \mid \bbe[x_N] = \mu_f, \bbe[(x_N-\mu_f) (x_N-\mu_f)^{\top}] = \Sigma_{f} \right\}.
\end{equation}
The objective is to steer the trajectories of \eqref{eqn_dt_stochastic_sys} from $x_0 \sim \PP_{x_{0}} \in \calp^{x_0}$ to $x_N \sim \PP_{x_{N}} \in \calp^{x_N}$ in $N$ time steps. 
This covariance steering objective is usually achieved by minimizing a cost function under specified convex state and input constraints. 
We define the concatenated variables required for the problem formulation as follows
\begin{align*}
    \mathbf{X} &= \begin{bmatrix} x^{\top}_0 & x^{\top}_1 & \dots & x^{\top}_N \end{bmatrix}^{\top}, \\
    \mathbf{U} &= \begin{bmatrix} u^{\top}_0 & u^{\top}_1 & \dots & u^{\top}_{N-1} \end{bmatrix}^{\top}, \\
    \mathbf{W} &= \begin{bmatrix} w^{\top}_0 & w^{\top}_1 & \dots & w^{\top}_{N-1} \end{bmatrix}^{\top}, \\
    E_{k} &= \begin{bmatrix} 0_{n, kn} & I_n & 0_{n, (N-k)n} \end{bmatrix}, \\
    \bar{\mathbf{X}} &= \begin{bmatrix} \mu^{\top}_{0} & \mu^{\top}_{1} & \dots & \mu^{\top}_{N} \end{bmatrix}^{\top}.
\end{align*}
Using the concatenated variables, \eqref{eqn_dt_stochastic_sys} can be written as 
\begin{align} \label{eqn_concat_sys}
    \mathbf{X} = \mathcal{A} x_0 + \mathcal{B} \mathbf{U} + \mathcal{D} \mathbf{W},
\end{align}
where the matrices $\mathcal{A}, \mathcal{B}$ and $\mathcal{D}$ are of appropriate dimensions containing the time varying system matrices $A_k, B_k$, and $D_k$ respectively. See \cite{pilipovsky2021covariance} for more details on this transformation.
It is evident that  $\bbp_{\mathbf{X}}$ is not known exactly but a similar ambiguity set like $\mathcal{P}^{x_{0}}$ can be constructed 
for $\bbp_{\mathbf{X}}$\footnote{Note that $\bbp_{\mathbf{X}}$ is the joint distribution formed with $\bbp_{x_{k}}, k = 0,\dots,N$ being its marginal distributions.}. 
That is,
\begin{align}
\mathcal{P}^{\mathbf{X}} := \left \{ \bbp_{\mathbf{X}} \mid \mathbb{E}[\mathbf{X}] = \bar{\mathbf{X}}, \mathbb{E}[(\mathbf{X} - \bar{\mathbf{X}})(\mathbf{X} - \bar{\mathbf{X}})^{\top}] = \Sigma_{\mathbf{X}} \right\}.
\end{align}
The cost function to optimize is given as follows
\begin{align} \label{eqn_concat_opt_cost}
    J(\mathbf{U}) := \mathbb{E} \left[ \mathbf{X}^{\top} \Bar{Q} \mathbf{X} + \mathbf{U}^{\top} \Bar{R} \mathbf{U} \right], 
\end{align}
where $\Bar{Q} = \mathbf{diag}(Q_0, Q_1, \dots, Q_{N-1})$ is the state penalty matrix and $\Bar{R} = \mathbf{diag}(R_0, R_1, \dots, R_{N-1})$ is the control penalty matrix with each $Q_k \succeq 0, R_{k} \succ 0$ for all $k \in [0, N-1]$.
Over the whole time horizon, we want the states to respect certain state constraints. 
Since the state is stochastic, constraint violation can be imposed through a chance constraint formulation. 
Specifically, a distributionally robust risk constraint formulation is preferred in this setting as the system state need not be Gaussian at any time step $k$.
Hence, we impose the following joint distributionally robust risk constraint that limits the worst case probability defined by $\bbp_{\mathbf{X}}$ 
of state constraint violation to be less than a pre-specified threshold, 
\begin{align} \label{eqn_dr_risk_constraint}
\sup_{\bbp_{\mathbf{X}} \in \mathcal{P}^{\mathbf{X}}} \, \bbp_{\mathbf{X}} \left( \bigwedge_{k=0}^{N} E_{k} \mathbf{X} \notin \mathcal{X}_{p} \right) \leq \Delta, 
\end{align}
where, $\Delta \in (0, 0.5]$ denotes the pre-specified total risk budget and $\mathcal{X}_{p}$ denotes the state constraint set to be satisfied. It is assumed to be a convex polytope and so can be represented using intersection of finite number of half-spaces as
\begin{align}
    \mathcal{X}_{p} := \bigcap^{M}_{i = 1} \left \{ x \in \mathbb{R}^{n} \mid a^{\top}_{i} x \leq b_{i} \right\},
\end{align}
where $a_{i} \in \mathbb{R}^{n}$ and $b_{i} \in \mathbb{R}$. Later in this article, we will extend the case for $\mathcal{X}_{p}$ being a more general convex cone constraint.

\begin{problem} \label{problem_1}
Given the system \eqref{eqn_concat_sys} and a total risk budget $\Delta$, we seek an optimal feedback control policy $\Pi^{\star} =[\pi^{\star}_0,..., \pi^{\star}_{N-1}]$ such that the control inputs $u^{\star}_k = \pi^{\star}_{k}(x_k)$, $k \in [0,N-1]$ steers the system from the $\mathbb{P}_{x_{0}}$ belonging to \eqref{eqn_ambig_x_0} to the $\mathbb{P}_{x_{N}}$ belonging to \eqref{eqn_ambig_x_N} while minimizing the finite-horizon cost function \eqref{eqn_concat_opt_cost} and by respecting the distributionally robust joint risk constraint given in \eqref{eqn_dr_risk_constraint}.
\end{problem}


\section{Covariance Steering with Distributionally Robust Risk Allocation} 
\label{sec_dr_ira}

In this section, we describe how to convert the joint distributionally robust risk constraint into individual distributionally robust risk constraints and use this result to steer both the mean and covariance of the initial state to the desired final mean and covariance. 

\subsection{Propagation of Mean and Covariance}

\noindent We adopt the following control policy from \cite{pilipovsky2021covariance} as follows, 
\begin{align} 
    \mathbf{U} &= \mathbf{V} + \mathbf{K} \mathbf{Y}, \quad \text{where} \label{eqn_control_law} \\
    \mathbf{Y} &= \mathcal{A} \underbrace{(x_0 - \mu_0)}_{:= y_{0}} + \mathcal{D} \mathbf{W}.
\end{align}
Here $\mathbf{V} = \begin{bmatrix} v^{\top}_0 & v^{\top}_1 & \dots & v^{\top}_{N-1} \end{bmatrix}^{\top} \in \mathbb{R}^{Nm}$ and $v_{k} \in \mathbb{R}^{m}$ for all $k = [0,N-1]$. 
The concatenated gain matrix $\mathbf{K} \in \mathbb{R}^{Nm \times (N+1)n}$ contains the individual gain matrices $K_k \in \mathbb{R}^{m \times n}$. The mean and covariance of $\mathbf{Y}$ are given by
\begin{align}
    \Bar{\mathbf{Y}} &= \mathbb{E}[\mathbf{Y}] = \mathbb{E}[\mathcal{A} y_0 + \mathcal{D} \mathbf{W}] = 0, \quad \text{and} \label{eqn_ybar}\\
    \Sigma_{\mathbf{Y}} &= \mathcal{A} \Sigma_{0} \mathcal{A}^{\top} + \mathcal{D} \Sigma_{\mathbf{W}} \mathcal{D}^{\top}, \label{eqn_sigma_y}
\end{align}
where $\Sigma_{\mathbf{W}} = \mathbf{diag}(\Sigma_{w}, \dots, \Sigma_{w}) \in \mathbb{R}^{Nr \times Nr}$. Substituting \eqref{eqn_ybar} and \eqref{eqn_sigma_y} into \eqref{eqn_control_law}, we can infer that 
\begin{align} \label{eqn_control_mean}
    \Bar{\mathbf{U}} &= \mathbb{E}[\mathbf{U}] = \mathbb{E}[\mathbf{V}] + \mathbb{E}[\mathbf{KY}] = \mathbf{V}, \\
    \Sigma_{\mathbf{U}} &= \mathbf{K} \Sigma_{\mathbf{Y}} \mathbf{K^{\top}}.
\end{align}
Similarly, substituting \eqref{eqn_control_law} and  \eqref{eqn_ybar} into \eqref{eqn_concat_sys}, the dynamics of the state mean $\bar{\mathbf{X}} := \mathbb{E}[\mathbf{X}]$, and the state covariance $\Sigma_{\mathbf{X}} = \mathbb{E} \left[ (\mathbf{X} - \bar{\mathbf{X}}) (\mathbf{X} - \bar{\mathbf{X}})^{\top} \right]$ can be written as 
\begin{align}
    \bar{\mathbf{X}} &= \mathcal{A} \mu_0 + \mathcal{B} \mathbf{V}, \label{eqn_mean_prop}\\
    \Sigma_{\mathbf{X}} &= (I + \mathcal{B} \mathbf{K}) \Sigma_{\mathbf{Y}} (I + \mathcal{B} \mathbf{K})^{\top}. \label{eqn_cov_prop}
\end{align}
Note that the initial and the terminal state moments can be expressed as follows
\begin{align}
    \mu_{0} &= E_{0} \bar{\mathbf{X}}, \quad \Sigma_{0} = E_{0} \Sigma_{\mathbf{X}} E_{0}, \quad \text{and} \label{eqn_initial_moments}\\
    \mu_{f} &= E_{N} \bar{\mathbf{X}}, \quad \Sigma_{f} = E_{N} \Sigma_{\mathbf{X}} E_{N}. \label{eqn_terminal_moments}
\end{align}
It is evident from \eqref{eqn_mean_prop} and \eqref{eqn_cov_prop} that the component $\mathbf{V}$ of the control law \eqref{eqn_control_law} steers the mean of the system from $\mu_0$ to $\mu_{f}$ and the component $\mathbf{K}$ of the control law \eqref{eqn_control_law} steers the covariance from $\Sigma_0$ to $\Sigma_{f}$. In order to make the problem convex, we relax the terminal covariance constraint in \eqref{eqn_terminal_moments} as an inequality constraint $\Sigma_{f} \succeq E_{N} \Sigma_{\mathbf{X}} E_{N}$ and subsequently reformulate it as a linear matrix inequality (LMI) using the Schur complement as
\begin{align} \label{eqn_terminal_cov_LMI}
    \begin{bmatrix} \Sigma_{f} & E_{N} (I + \mathcal{B} \mathbf{K}) \Sigma^{\frac{1}{2}}_{\mathbf{Y}} \\
    \Sigma^{\frac{1}{2}}_{\mathbf{Y}} (I + \mathcal{B} \mathbf{K})^{\top} E^{\top}_{N} & I
    \end{bmatrix} \succeq 0.
\end{align}
Note that the cost given by \eqref{eqn_concat_opt_cost} can be decoupled into the cost on the mean and the cost on the covariance as follows
\begin{align}
    J(\mathbf{V}, \mathbf{K}) &= \underbrace{\Bar{\mathbf{X}}^{\top} \Bar{Q} \Bar{\mathbf{X}} + \Bar{\mathbf{U}}^{\top} \Bar{R} \Bar{\mathbf{U}}}_{:= J_{\mu}} + \underbrace{\mathbf{tr}\left( \Bar{Q} \Sigma_{\mathbf{X}} + \Bar{R} \Sigma_{\mathbf{U}} \right)}_{:= J_{\Sigma}}.
\end{align}
Substituting \eqref{eqn_mean_prop} and \eqref{eqn_cov_prop}, we get
\begin{align} \label{eqn_new_cost}
    J(\mathbf{V}, \mathbf{K}) = \, &(\mathcal{A} \mu_0 + \mathcal{B} \mathbf{V})^{\top} \Bar{Q} (\mathcal{A} \mu_0 + \mathcal{B} \mathbf{V}) + V^{\top} \Bar{R} V + \nonumber \\
    &\mathbf{tr}\left[\left( (I + \mathcal{B} \mathbf{K})^{\top} \Bar{Q} (I + \mathcal{B} \mathbf{K}) + \mathbf{K}^{\top} \Bar{R} \mathbf{K} \right) \Sigma_{\mathbf{Y}} \right].
\end{align}

\subsection{Distributionally Robust Polytopic Joint Risk Constraints}

Given that the state constraint set $\mathcal{X}_{p}$ is assumed to be a convex polytope, the worst case joint probability of violating any of the $M$ state constraints over the horizon $N$ given by \eqref{eqn_dr_risk_constraint} can be equivalently written as
\begin{align} 
\sup_{\bbp_{\mathbf{X}} \in \mathcal{P}^{\mathbf{X}}} \, \bbp_{\mathbf{X}} \left( \bigwedge_{k=0}^{N} E_{k} \mathbf{X} \notin \mathcal{X}_{p} \right) \leq \Delta 
\end{align}
or
\begin{align} \label{eqn_dr_risk_constraint_polytope}
\sup_{\bbp_{\mathbf{X}} \in \mathcal{P}^{\mathbf{X}}} \, \bbp_{\mathbf{X}} \left( \bigwedge_{k=1}^{N} \bigwedge_{i=1}^{M} a^{\top}_{i} E_{k} \mathbf{X} > b_{i} \right) \leq \Delta.
\end{align}
Using Boole's inequality, the above joint distributionally robust risk constraint can be decomposed into individual distributionally robust risk constraints at each time step with $\delta_{i,k}$ denoting the individual risk bound\footnote{The first and second subscript in $\delta_{i,k}$ denote the constraint defining the state constraint set $\mathcal{X}_{p}$ and the time step respectively.} representing the worst case probability of violating the $i\ts{th}$ state constraint at time step $k$. 
That is, for each time step $k \in [1,N]$, and for each half-space $i \in [1,M]$ defining the constraint set $\mathcal{X}_{p}$, we have 
\begin{align} \label{eqn_individual_dr_risk_constraintA}
\sup_{\bbp_{x_{k}} \in \mathcal{P}^{x_{k}}} \, \bbp_{x_{k}} \left( a^{\top}_{i} E_{k} \mathbf{X} > b_{i} \right) &\leq \delta_{i,k},
\end{align}
or
\begin{align} \label{eqn_individual_dr_risk_constraint}
\inf_{\bbp_{x_{k}} \in \mathcal{P}^{x_{k}}} \, \bbp_{x_{k}} \left( a^{\top}_{i} E_{k} \mathbf{X} \leq b_{i} \right) &\geq 1 - \delta_{i,k}, \\
\sum_{k=1}^{N} \sum_{i=1}^{M} \delta_{i,k} &\leq \Delta. 
\end{align}

Using Cantelli's inequality, the individual distributionally robust risk constraint in \eqref{eqn_individual_dr_risk_constraint} can be equivalently reformulated as deterministically tightened convex second-order cone constraint on the state mean as described in \cite{dr_cc_lp}. That is, 
\begin{equation} \label{eqn_dr_constraint_second_cone_reform}
a^{\top}_{i} E_{k} \Bar{\mathbf{X}} \leq b_{i} - \underbrace{\sqrt{\frac{1-\delta_{i,k}}{\delta_{i,k}}}}_{:= \mathcal{Q}(1-\delta_{i,k})} \left \| \Sigma^{\frac{1}{2}}_{\mathbf{Y}} (I + \mathcal{B} \mathbf{K})^{\top} E^{\top}_{k} a_{i} \right \Vert_{2},
\end{equation}
where the DR quantile function $\mathcal{Q}(\delta_{i,k}) := \sqrt{\delta_{i,k}/(1 - \delta_{i,k})}$ plays a similar role to that of $\mathbf{\Phi}^{-1}$ corresponding to the Gaussian case. 
%
%
Note that, $\mathcal{Q}(1-\delta_{i,k})$ is also a monotonically increasing function of the risk, just like $\mathbf{\Phi}^{-1}$ as shown in Figure \ref{fig_tighteningcompare} and the deterministic constraint tightening defined using it leads to a stronger tightening than the constant associated with the Gaussian chance constrained tightening. This stronger tightening will ensure that the worst case probability of state constraint violation is satisfied for any arbitrary distribution in the ambiguity set. It is clear from \eqref{eqn_dr_constraint_second_cone_reform} that Problem 1 can now be converted into the following convex programming problem.
\begin{figure}
    \centering
    \includegraphics[scale=0.14]{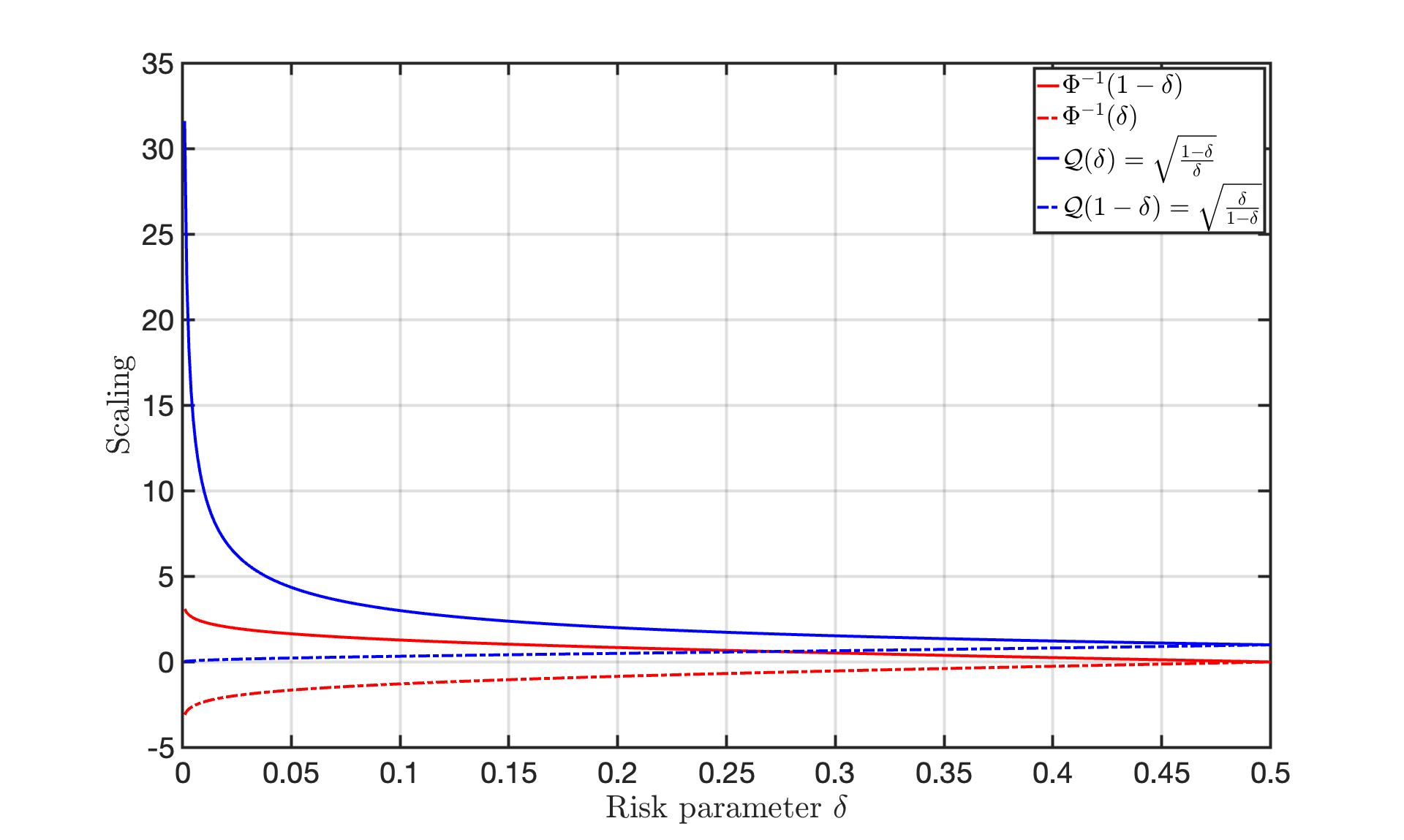}
    \caption{Comparison of tightening constant for the Gaussian case and the distributionally robust case using the Cantelli's inequality are shown here for the risk parameter $\delta \in (0,0.5]$.}
    \label{fig_tighteningcompare}
\end{figure}

\begin{problem}\label{problem_2}
Given the system \eqref{eqn_concat_sys} and a total risk budget $\Delta$, we seek an optimal feedback control sequence of inputs $\mathbf{V}^{\star}, \mathbf{K}^{\star}$ that steers the system from the initial state distribution belonging to \eqref{eqn_ambig_x_0} with moments given by \eqref{eqn_initial_moments} to the final state distribution belonging to \eqref{eqn_ambig_x_N} with moments as in \eqref{eqn_terminal_moments} and \eqref{eqn_terminal_cov_LMI} by minimizing the finite-horizon cost function \eqref{eqn_new_cost} and by respecting the DR risk constraint tightening given in \eqref{eqn_dr_constraint_second_cone_reform}. 
\end{problem}

\subsection{Distributionally Robust Risk Allocation Optimization}

From Theorem~1 in \cite{pilipovsky2021covariance}, the optimal cost obtained from solving Problem~2 will be a monotonically decreasing function of the stage risk budget $\delta_{i,k}$. 
For brevity of notation, we define the vector of all individual risk bounds over the whole time horizon and across all half-spaces defining the state constraint set $\mathcal{X}_{p}$ as $\bm{\delta} = \begin{bmatrix} \delta_{1,1} & \dots & \delta_{M,N} \end{bmatrix}^{\top} \in \mathbb{R}^{MN}$. 
Recall that in the risk allocation problem, the stage risk budget $\delta_{i,k}$ becomes a decision variable along with $\mathbf{K}$. 
However, in the distributionally robust risk constraints given by \eqref{eqn_dr_constraint_second_cone_reform}, $\delta_{i,k}$ and $\mathbf{K}$ occur in a bilinear form. 
A better tractable approach would be to concurrently allocate $\delta_{i,k}$ when solving the optimization Problem 2, so as to minimize the total cost given by \eqref{eqn_new_cost}. 

\subsubsection{Two-Stage Optimization Framework} 

In the following discussion, we elaborate how to optimally allocate the risk across the time steps and across the constraints defining the state constraint set $\mathcal{X}_{p}$. Following \cite{pilipovsky2021covariance}, a two-stage optimization framework is presented here. 
The upper stage optimization finds the optimal risk allocation $\bm{\delta^{\star}}$ and the lower stage solves the covariance steering problem given by Problem~2 for the optimal controller $\mathbf{U^{\star}}$ given the optimal risk allocation $\bm{\delta^{\star}}$ from the upper stage. 

Let the value of the objective function after the lower stage optimization for a given risk allocation $\bm{\delta}$ be 
\begin{align}
    J^{\star}(\bm{\delta}) := \min_{\mathbf{V}, \mathbf{K}} J(\mathbf{V}, \mathbf{K}).
\end{align}
Then the upper-stage optimization problem can then be formulated as follows:
\begin{equation}
\begin{aligned} \label{eqn_upper_stage_dr_ira}
\underset{\bm{\delta}}{\textrm{minimize}} \quad & J^{\star}(\bm{\delta})\\
\textrm{subject to} \quad & \sum_{k=1}^{N} \sum_{i=1}^{M} \delta_{i,k} \leq \Delta, \\
& \delta_{i,k} \geq 0. 
\end{aligned}
\end{equation}
Note that \eqref{eqn_upper_stage_dr_ira} is a convex optimization problem, given that the objective function $J^{\star}(\bm{\delta})$ is convex, and $\Delta \in (0, 0.5]$. Following Theorem 1 in \cite{pilipovsky2021covariance}, the optimal cost can be reduced with each successive iteration by carefully increasing the risk allocations $\delta_{i,k}$. That is, the risk can be lowered by tightening the constraints that are too conservative, and increased by loosening the constraints that are already active. 
It now remains to define active and inactive constraints in the context of distributionally robust risk allocation. Note that the distributionally robust risk constraint given by \eqref{eqn_dr_constraint_second_cone_reform} can be equivalently written as
\begin{equation} \label{eqn_delta_ik}
\footnotesize
    \delta_{i,k} 
    \geq \left(1 + \left(\frac{b_{i} - a^{\top}_{i} E_{k} \Bar{\mathbf{X}}^{\star}} {\left \| \Sigma^{\frac{1}{2}}_{\mathbf{Y}} (I + \mathcal{B} \mathbf{K}^{\star})^{\top} E^{\top}_{k} a_{i} \right \Vert_{2}}\right)^{2} \right)^{-1} =: \bar{\delta}_{i,k}.
\end{equation}
Here the quantity $\bar{\delta}_{i,k}$ represents the true risk experienced by the optimal trajectories, when using $(\mathbf{V}^{\star}, \mathbf{K}^{\star})$. 
Clearly, the selected risk  need not be equal to the actual risk once the optimization is completed. 
When $\delta_{i,k} = \bar{\delta}_{i,k}$, we say that the constraint \eqref{eqn_delta_ik} is active, and is inactive otherwise. Solutions are termed good when the true risk is within a small margin of the allocated risk and the conservative ones correspond to the cases when $\delta_{i,k} > \bar{\delta}_{i,k}$.

\subsubsection{Distributionally Robust Iterative Risk Allocation (DR-IRA) Algorithm}

Starting with some feasible risk allocation $\bm{\delta}^{(j)}$, with $j$ denoting the iteration number, solve Problem \ref{problem_2} to get the optimal controller $(\mathbf{V}^{\star}_{(j)}, \mathbf{K}^{\star}_{(j)})$. 
It is straightforward to observe that using the above optimal controller at iteration $j$ leads to the optimal mean trajectory $\mathbf{\bar{\mathbf{X}}^{\star}_{(j)}}$ respecting the optimal stage risk budget $\delta^{(j)}_{i, k}$. 
%
The risk budget is then successively loosened and tightened according to Algorithm~\ref{alg_drira} as in  \cite{pilipovsky2021covariance}, with the little change in the procedure that we solve the Problem \ref{problem_2} in this paper instead of the Problem 2 mentioned in Algorithm~1 of  \cite{pilipovsky2021covariance} and this new allocation is fed back to the optimizer.
%
This iterative process generates a sequence of risk allocations by continuously lowering the optimal cost. 

\begin{algorithm}
    \caption{The covariance steering algorithm with DR-IRA }\label{alg_drira}
    \begin{algorithmic}
        \STATE \textbf{Input:} $\delta^{(j)}_{k} \gets \Delta / (NM), \epsilon, \rho$
        \STATE \textbf{Output:} $\delta^{\star}, J^{\star}, \mathbf{V}^{\star}, \mathbf{K}^{\star}$
        \WHILE{$ \left | J^{\star} - J^{\star}_{\mathrm{prev}} \right \vert > \epsilon$}
            \STATE $J^{\star}_{\mathrm{prev}} \gets J^{\star}$
            \STATE Solve \textbf{Problem \ref{problem_2}} with current $\delta$ to get $\bar{\delta}$.
            \STATE $\hat{N} \gets$ the number of active constraints
            \IF{$\hat{N} = 0$ or $\hat{N} = MN$}
                \STATE break
            \ENDIF
            \FOR{\textbf{each} $j$\textsuperscript{th} inactive constraint at $k$\textsuperscript{th} time step}
                \STATE $\delta^{(j)}_{k} \gets \rho \delta^{(j)}_{k} + (1-\rho) \bar{\delta}^{(j)}_{k}$
            \ENDFOR
            \STATE $\delta_{\mathrm{res}} \gets \Delta - \sum^{N}_{k=1} \sum^{M}_{j=1} \delta^{(j)}_{k}$
            \FOR{\textbf{each} $j$\textsuperscript{th} active constraint at $k$\textsuperscript{th} time step}
                \STATE $\delta^{(j)}_{k} \gets \delta^{(j)}_{k} + \delta_{\mathrm{res}}/\hat{N}$
            \ENDFOR
        \ENDWHILE
    \end{algorithmic}
\end{algorithm}

\subsection{Distributionally Robust Convex Conic Risk Constraints}
\label{sec:cone_constraints}
Constraints of the convex cone forms are usually known to be more prevalent than the polytopic constraints in engineering applications such as spacecraft rendezvous or landing. In the following discussion, we extend our covariance steering formulation to distributionally robust convex cone constraints. Consider the following convex cone state constraint set 
\begin{align} \label{eqn_cone_constraint}
    \mathcal{X}_{c} := \left\{ x \in \mathbb{R}^{n} \mid \norm{Ax+b}_{2} \leq c^{\top}x+d \right\}.
\end{align}
We can specify the distributionally robust risk constraint for all time steps $k \in [1,N]$ with conic state constraint set $\mathcal{X}_{c}$ as 
\begin{align}
    \sup_{\bbp_{x_{k}} \in \mathcal{P}^{x_{k}}} \, \bbp_{x_{k}} \left[ \norm{Ax_{k}+b}_{2} \leq c^{\top}x_{k} + d \right] \geq 1 - \delta_{k},
\end{align}
or
\begin{align}  
\sup_{\bbp_{x_{k}} \in \mathcal{P}^{x_{k}}} \, \bbp_{x_{k}} \left[x_{k} \in \mathcal{X}_{c}\right] &\geq 1 - \delta_{k}, \label{eqn_cone_risk_constraint}\\
\sum^{N}_{k=1} \delta_{k} &\leq \Delta. \label{eqn_dr_cone_risk_budget}
\end{align}
Notice that \eqref{eqn_cone_risk_constraint} is an infinite dimensional constraint and not necessarily convex~\cite{pilipovsky2021covariance}. 
Hence, we resort to a convex approximation so that \eqref{eqn_cone_risk_constraint} and \eqref{eqn_dr_cone_risk_budget} holds true for all $\Delta \in (0,0.5]$.
We first seek to relax \eqref{eqn_cone_risk_constraint} as DR quadratic risk constraint and then use the reverse union bound approximation.

\subsubsection{Relaxing the DR Conic Risk Constraint}
\begin{lemma} \label{lemma_dr_cone_relax}
Given $\delta_{k} \in (0, 0.5]$ for all  $k \in [1,N]$, the following DR quadratic risk constraint 
\begin{align} \label{eqn_dr_quadratic_risk_constraint_main}
    \sup_{\bbp_{x_{k}} \in \mathcal{P}^{x_{k}}} \, \bbp_{x_{k}} \left[ \norm{Ax_{k}+b}_{2} \leq c^{\top} \bar{x}_{k} + d \right] \geq 1 - \delta_{k},
\end{align}
is a relaxation of the original DR conic risk constraint
\begin{align} \label{eqn_original_dr_quadratic_risk_constraint}
    \sup_{\bbp_{x_{k}} \in \mathcal{P}^{x_{k}}} \, \bbp_{x_{k}} \left[ \norm{Ax_{k}+b}_{2} \leq c^{\top} x_{k} + d \right] \geq 1 - \delta_{k}.
\end{align}
\end{lemma}
\begin{proof}
See Appendix~1.
\end{proof}

For each time step $k \in [1,N]$, denote $\psi_{i,k} := a^{\top}_i x_{k} + b_{i}$ and $\kappa_{k} := c^{\top} \bar{x}_{k} + d$ with $a^{\top}_{i}$ denoting the i\textsuperscript{th} row of $A$, and observe that \eqref{eqn_original_dr_quadratic_risk_constraint} can be equivalently written as
\begin{align} 
    \sup_{\bbp_{x_{k}} \in \mathcal{P}^{x_{k}}} \, \bbp_{x_{k}} \left[ \left( \sum^{n}_{i = 1} \psi^{2}_{i,k} \right)^{\frac{1}{2}} \leq \kappa_{k} \right] \geq 1 - \delta_{k},
    \end{align}
    or
    \begin{align} 
\sup_{\bbp_{x_{k}} \in \mathcal{P}^{x_{k}}} \, \bbp_{x_{k}} \left[ \sum^{n}_{i = 1} \psi^{2}_{i,k} \leq \kappa^{2}_{k} \right] \geq 1 - \delta_{k} \label{eqn_dr_quadratic_risk_constraint}
\end{align}

\begin{proposition}
The DR quadratic constraint \eqref{eqn_dr_quadratic_risk_constraint} is satisfied if the following constraints are satisfied (subscript $k$ dropped for brevity of notation) for some non-negative $f_1, \dots, f_n$ and $\beta_1,\dots,\beta_n$:
\begin{subequations}
\begin{align}
    \sup_{\bbp_{x} \in \mathcal{P}^{x}} \, \bbp_{x} \left[ \sum^{n}_{i = 1} \left | \psi_{i} \right \vert \leq f_i \right] &\geq 1 - \beta_i \delta, \quad i = [1,N], \label{eqn_main_dr_quad_constraint}\\
    \sum^{n}_{i=1} f^{2}_{i} &\leq \kappa^{2}, \label{eqn_main_dr_quad_constraint1}\\
    \sum^{n}_{i=1} \beta_{i} &= 1.
\end{align}
\end{subequations}
\end{proposition}

\begin{proof}
The proof uses the same arguments as in \cite{pilipovsky2021covariance} and hence is omitted.
\end{proof}

\subsubsection{Approximation Using Reverse Union Bound} 
Note that the constraints given by \eqref{eqn_main_dr_quad_constraint1} can be equivalently written using the mean dynamics given by \eqref{eqn_mean_prop} as
\begin{align}
    \norm{f_{k}}_{2} \leq \kappa_{k}, \quad k = 1, \dots,N. 
 \end{align}
 which holds if and only if
 \begin{align}
  \norm{f_{k}}_{2} \leq c^{\top} E_{k} (\mathcal{A} \mu_0 + \mathcal{B} \mathbf{V}) + d, \quad k = 1, \dots,N. \label{eqn_main_dr_quad_constraint1_relax}
\end{align}

\begin{proposition}
Let $\epsilon^{1}_{i,k}, \epsilon^{2}_{i,k} > 0$ for all $i = 1,\dots,n$ and $k = 1,\dots,N$. 
Assume that the following convex DR SOC constraints hold true for some $\mathbf{V}, \mathbf{K}$, and $\epsilon^{1}_{i,k}+ \epsilon^{2}_{i,k} \geq 2 - \beta_{i} \delta_{k}$.
\begin{subequations}  \label{eqn_rubs}
\begin{align} 
a^{\top}_{i} E_{k} \Bar{\mathbf{X}} + \sqrt{\frac{\epsilon^{1}_{i,k}}{1-\epsilon^{1}_{i,k}}} \left \| \Sigma^{\frac{1}{2}}_{\mathbf{Y}} (I + \mathcal{B} \mathbf{K})^{\top} E^{\top}_{k} a_{i} \right \Vert_{2} \leq f_{i,k} - b_{i}, \label{eqn_rub_1}\\
-a^{\top}_{i} E_{k} \Bar{\mathbf{X}} + \sqrt{\frac{\epsilon^{2}_{i,k}}{1-\epsilon^{2}_{i,k}}} \left \| \Sigma^{\frac{1}{2}}_{\mathbf{Y}} (I + \mathcal{B} \mathbf{K})^{\top} E^{\top}_{k} a_{i} \right \Vert_{2} \leq f_{i,k} + b_{i}. \label{eqn_rub_2}
\end{align}
\end{subequations}
Then, the two-sided DR risk constraint given by \eqref{eqn_main_dr_quad_constraint} holds true as well.
\end{proposition}
\begin{proof}
See Appendix~2.
\end{proof}
It is important to fix $\epsilon^{1}_{i,k}$, and $\epsilon^{2}_{i,k}$ apriori to avoid bilinearity in the decision variables of constraints \eqref{eqn_rub_1} and \eqref{eqn_rub_2} respectively. Thus, the relaxed DR cone risk constraints given by \eqref{eqn_dr_quadratic_risk_constraint_main} was approximated using \eqref{eqn_main_dr_quad_constraint1_relax} and the alternate approximations \eqref{eqn_rub_1} and \eqref{eqn_rub_2} of the two-sided DR risk constraint \eqref{eqn_main_dr_quad_constraint}. Note that the constraints \eqref{eqn_main_dr_quad_constraint1_relax}, \eqref{eqn_rub_1} and \eqref{eqn_rub_2} are convex and can be solved using the standard SDP solvers used similarly for the polyhedral counterparts. 
In terms of computational complexity, notice that there are $2n+1$ SOC constraints per time step and thus $(2n+1)N$ SOC constraints in total.


\section{Numerical Simulations} 
\label{sec_num_sim}

In this section, we demonstrate the proposed approach using two examples. 
One dealing with spacecraft proximity operation as in \cite{pilipovsky2021covariance}, and the second one using a simple double integrator based path planning as in \cite{okamoto2018optimal}. The code is available at \url{https://github.com/venkatramanrenganathan/Covariance-Steering-With-Optimal-DR-Risk-Allocation}.
%
\subsection{Double Integrator Path Planning}

\subsubsection{Simulation Setup}

We consider a path-planning problem for a vehicle modelled using the following time invariant and stochastic double integrator system dynamics:
\begin{align*}
    A = \begin{bmatrix} I_{2} & \Delta t I_{2} \\ 0_{2 \times 2} & I_{2} \end{bmatrix}, \quad B = \begin{bmatrix} (\Delta t)^{2} I_{2} \\ \Delta t I_{2} \end{bmatrix},\quad D = 10^{-3} I_{4}.
\end{align*}
We assume $\mu_{0} = \begin{bmatrix} -10 & 1 & 0 & 0 \end{bmatrix}$ and $\Sigma_{0} = \texttt{diag}(0.1, 0.1, 0.01, 0.01)$ and the discretization time step to be $\Delta t = 0.2$, with the horizon $N = 15$. 
We wish to steer the distribution from the above initial state to the final mean $\mu_{f} = 0$ with final covariance $\Sigma_{f} = 0.25\, \Sigma_{0}$, while minimizing the cost function with penalty matrices $Q = \texttt{diag}(10, 10, 1, 1)$ and $R = 10^{3} I_{2}$. 
The state constraints are defined as $0.2(x-1) \leq y \leq -0.2(x-1)$. We impose the joint probability of failure over the whole horizon to be $\Delta = 0.10$, which implies that the worst case probability of violating any state constraint over the whole horizon is less than $10\%$. 
For the Monte Carlo trials, the disturbances were sampled from a multivariate Laplacian distribution with zero mean and unit covariance.
Similarly, the initial state $x_{0}$ was sampled from a multivariate Laplacian distribution with mean $\mu_{0}$ and covariance $\Sigma_{0}$. 
For comparison, simulations were performed with both Gaussian and distributionally-robust chance constraints.

\begin{figure*}
\centering
\begin{subfigure}[b]{0.7\textwidth}
   \includegraphics[width=\linewidth]{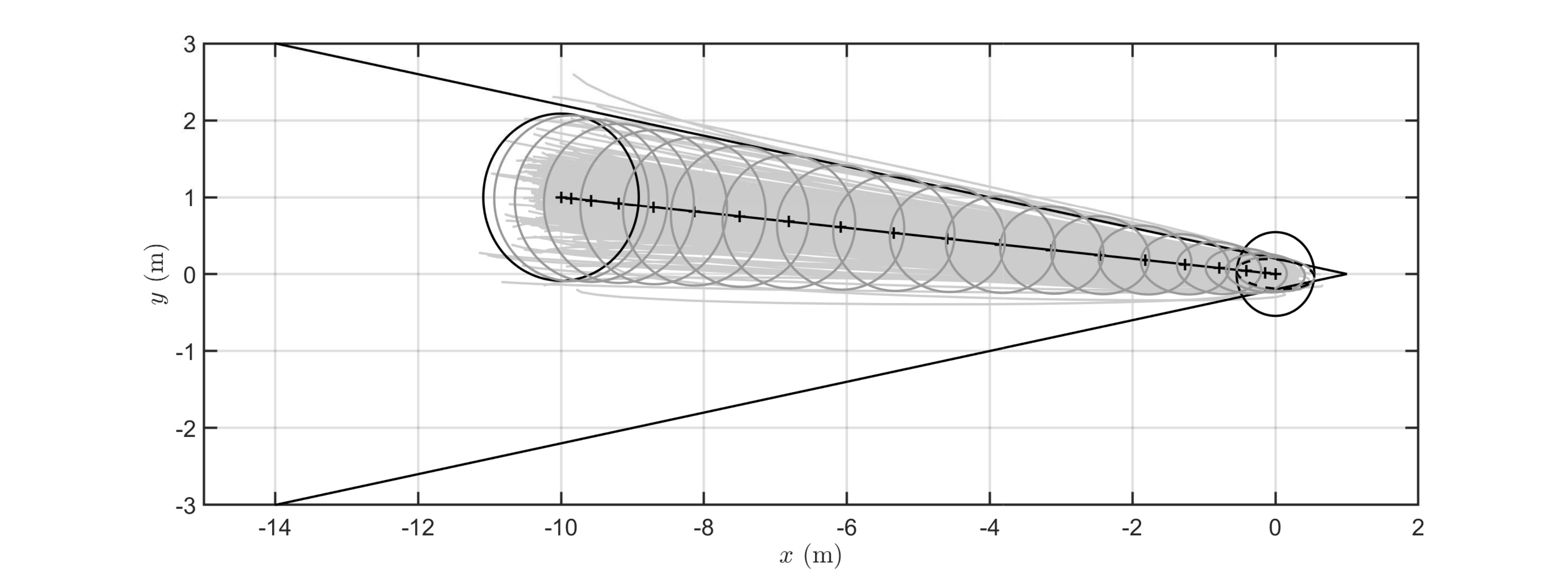}
   \caption{Gaussian chance constraints.}
   \label{fig:double_int_trajs_gaussian} 
\end{subfigure}

\begin{subfigure}[b]{0.7\textwidth}
   \includegraphics[width=\linewidth]{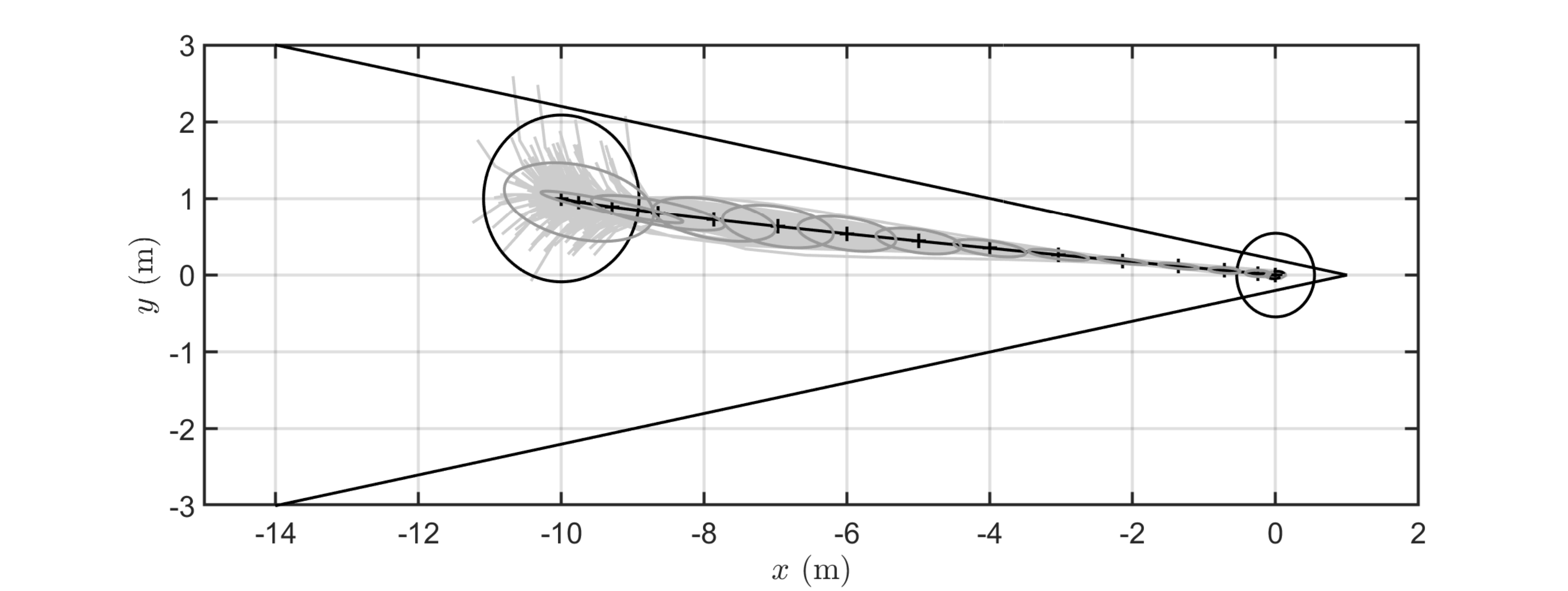}
   \caption{Distributionally-robust chance constraints.}
   \label{fig:double_int_trajs_DR}
\end{subfigure}
\caption{Comparison of 100 independent Monte Carlo trajectories using Gaussian (a) and DR (b) chance constraints with $x_{0}$ and noise history $\{w_k\}$ sampled from multivariate Laplacian distributions.}
\label{fig:double_int_trajs}
\end{figure*}

\begin{figure*}
\centering
\begin{minipage}[b]{.45\textwidth}
    \centering
    \includegraphics[scale=0.4]{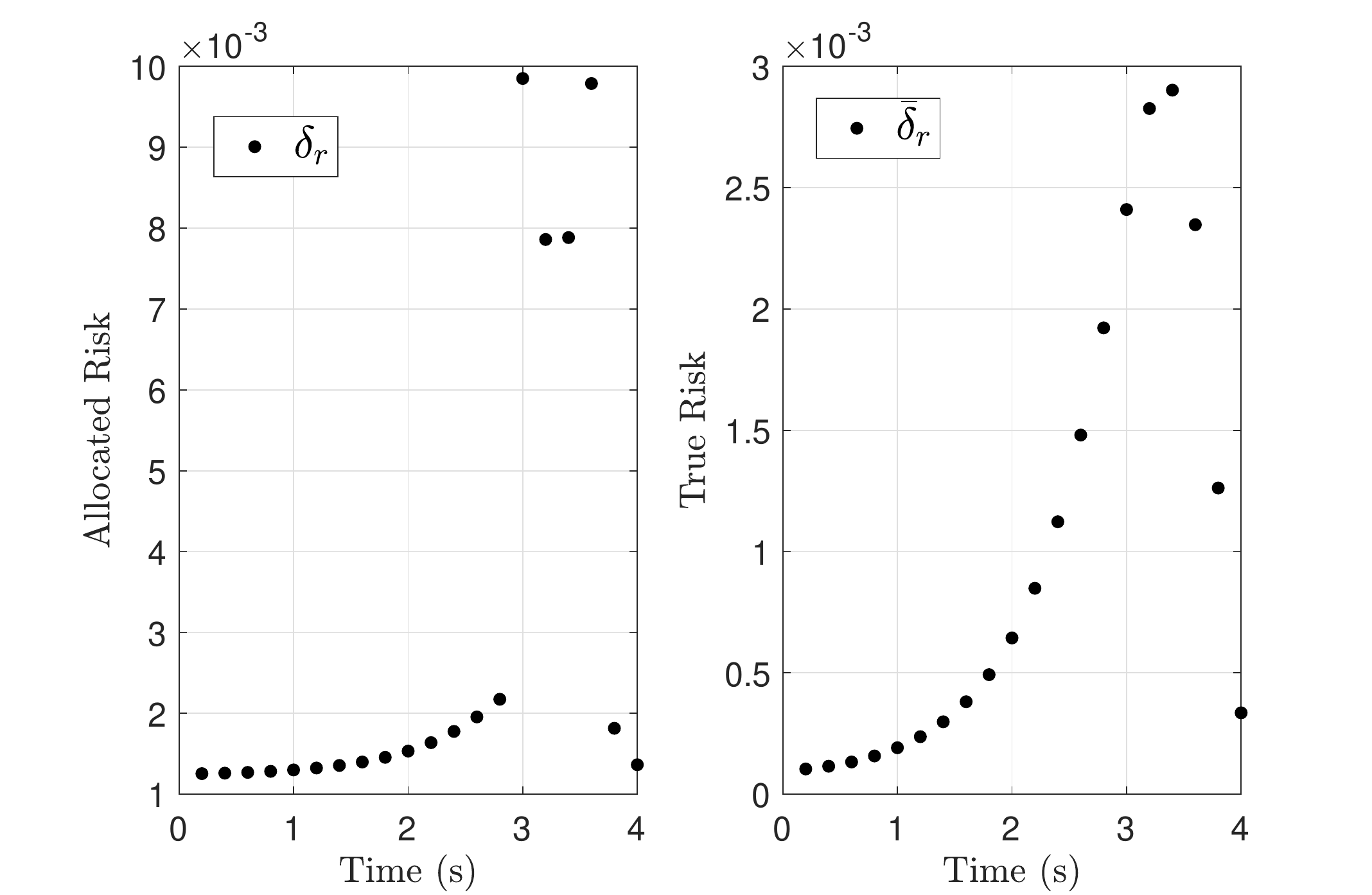}
    \caption{Allocated and true risks with Gaussian chance constraints.}
    \label{fig:CC_Risk_Allocate}
\end{minipage}
\hspace{1cm}
\begin{minipage}[b]{.45\textwidth}
    \centering
    \includegraphics[scale=0.37]{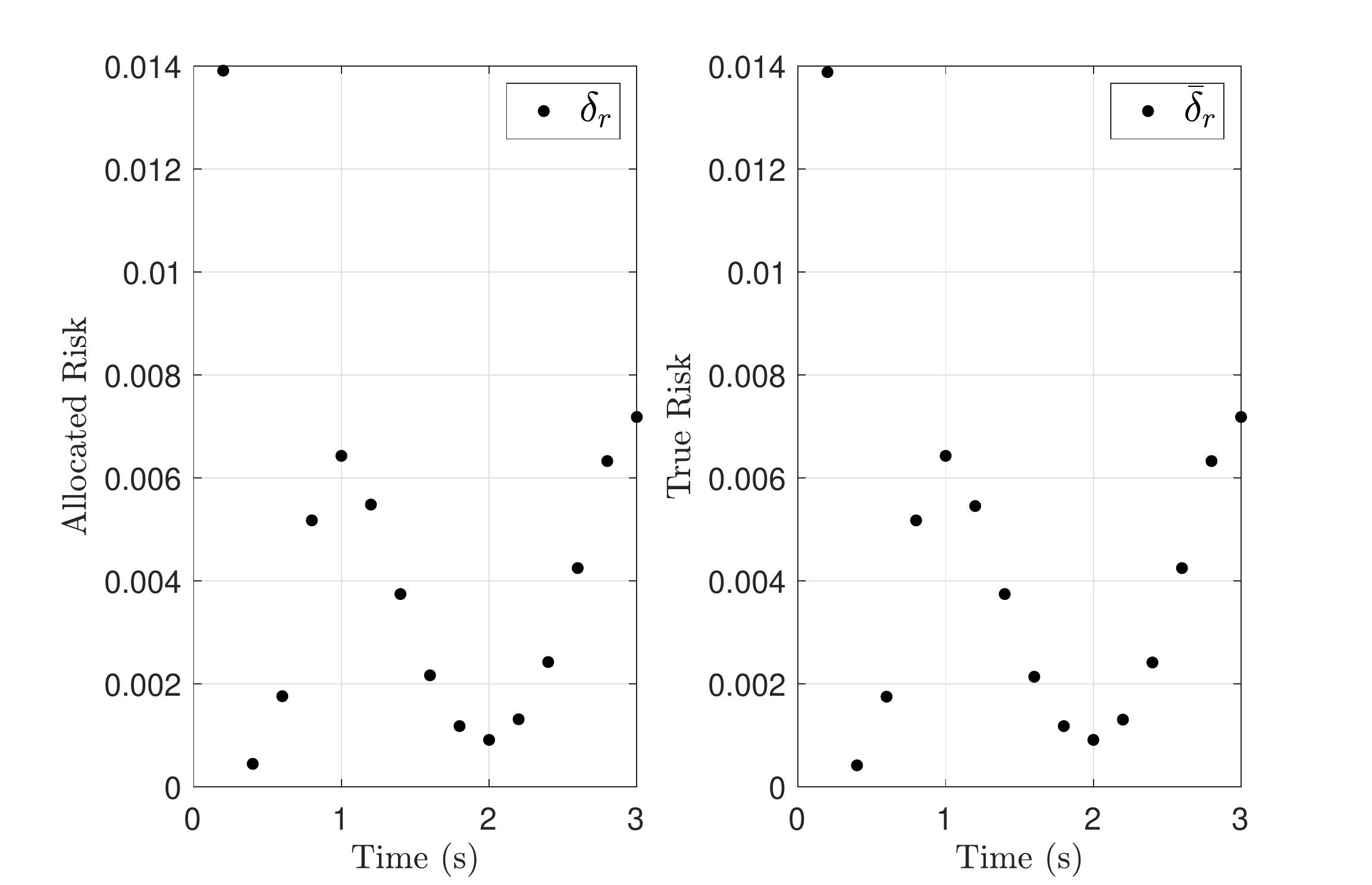}
    \caption{Allocated and true risks with distributionally-robust chance constraints.}
\label{fig:DR_Risk_Allocate}
\end{minipage}
\end{figure*}

\subsubsection{Results \& Discussion}

The results from 500 independent Monte-Carlo trials are shown in Figure~\ref{fig:double_int_trajs}. 
Given that the disturbances were sampled from a multivariate Laplacian distribution, it is evident from Figure~\ref{fig:double_int_trajs_gaussian} that some Monte-Carlo trials resulted in a violation of the state constraints as the covariance steering was performed assuming Gaussian disturbances. 
On the other hand, the distributionally robust risk-constrained covariance steering (Figure~\ref{fig:double_int_trajs_DR}) ensured that the total risk budget is respected, despite being more conservative.
Moreover, all probabilistic state constraints were satisfied. 
This shows that, assuming Gaussian chance constraints might potentially lead to severe miscalculation of the risk.
It is evident from Figures~\ref{fig:CC_Risk_Allocate} and \ref{fig:DR_Risk_Allocate} that the true risk is always upper bounded by the allocated risk regardless of whether a Gaussian or a DR iterative risk allocation is employed. 

\subsection{Spacecraft Proximity Operation}

\subsubsection{Simulation Setup}

We demonstrate the proposed theory using the spacecraft proximity operations problems in orbit as described in \cite{pilipovsky2021covariance} albeit with certain parameter changes. 
We perform three set of simulation here, namely: 
a) covariance steering with distributionally robust polytopic state risk constraints; 
b) covariance steering with distributionally robust iterative risk allocation for polytopic state risk constraints; 
and c) covariance steering with distributionally robust iterative risk allocation for convex conic state risk constraints using reverse union bound approximation. For the case of polytopic state risk constraints, we assume $\mu_{0} = \begin{bmatrix} 100 & -120 & 90 & 0 & 0 & 0 \end{bmatrix}$ and $\Sigma_{0} = 0.4 \texttt{diag}(1, 1, 1, 0.1, 0.1, 0.1)$. We wish to steer the distribution from the above initial state to the final mean $\mu_{f} = 0$ with final covariance $\Sigma_{f} = 0.5\, \Sigma_{0}$, while minimizing the cost function with penalty matrices $Q = \texttt{diag}(10, 10, 10, 1, 1, 1)$ and $R = 10^{3} I_{3}$. 
We impose the joint probability of failure over the whole horizon to be $\Delta = 0.15$. 
As in the previous example, the disturbances were sampled from a multivariate Laplacian distribution with zero mean and unit covariance. Similarly, the initial state $x_{0}$ was sampled from a multivariate Laplacian distribution with mean $\mu_{0}$ and covariance $\Sigma_{0}$.
For the cone constraints, we shift the initial $x$ mean to $\mu_0^{(1)} = 10$.

\subsubsection{Results \& Discussion}

\begin{figure*}
\centering
\begin{subfigure}[b]{.3\textwidth}
    \centering
    \hspace{-1cm}
    \includegraphics[scale=0.22]{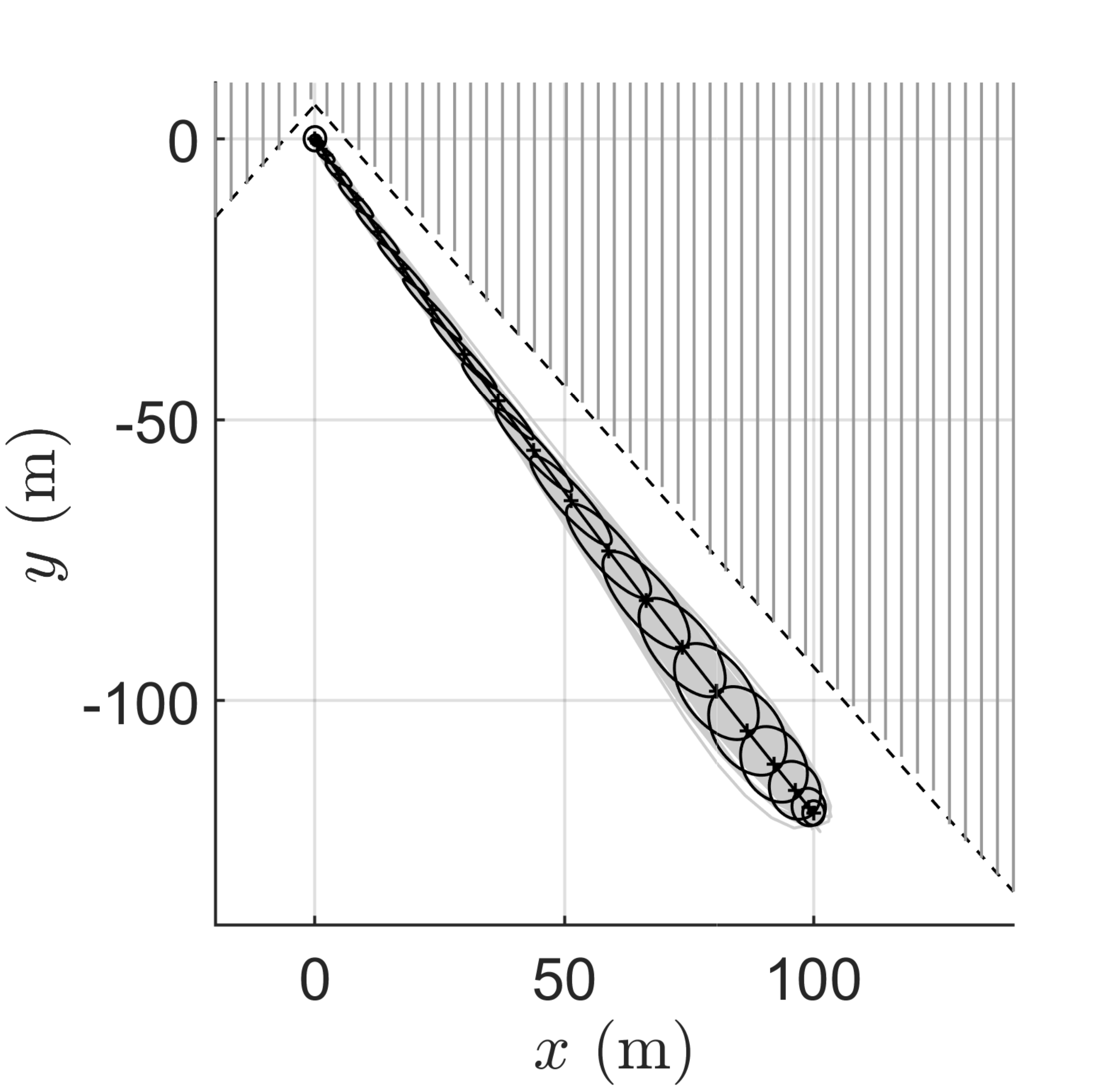}
    \caption{DR chance constraint solution with optimal risk allocation.}
    \label{fig:DR_SC_optimal}
\end{subfigure}
\hspace{2pt}
\begin{subfigure}[b]{.3\textwidth}
    \centering
    \hspace{-.8cm}
    \includegraphics[scale=0.31]{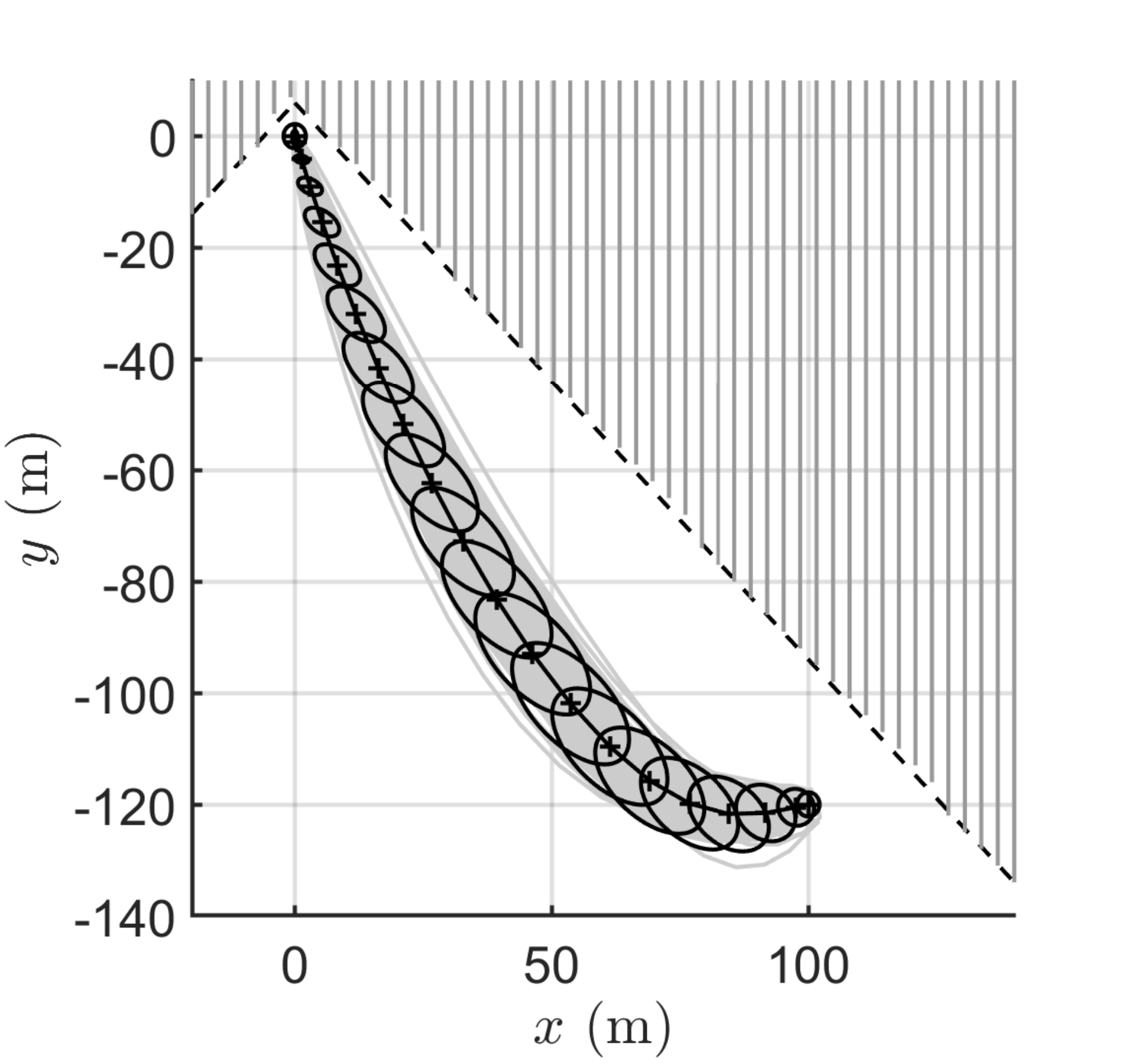}
    \caption{DR chance constraint solution with uniform risk allocation.}
    \label{fig:DR_SC_uniform}
\end{subfigure}
\hspace{2pt}
\begin{subfigure}[b]{.3\textwidth}
    \centering
    \hspace{-.5cm}
    \includegraphics[scale=0.28]{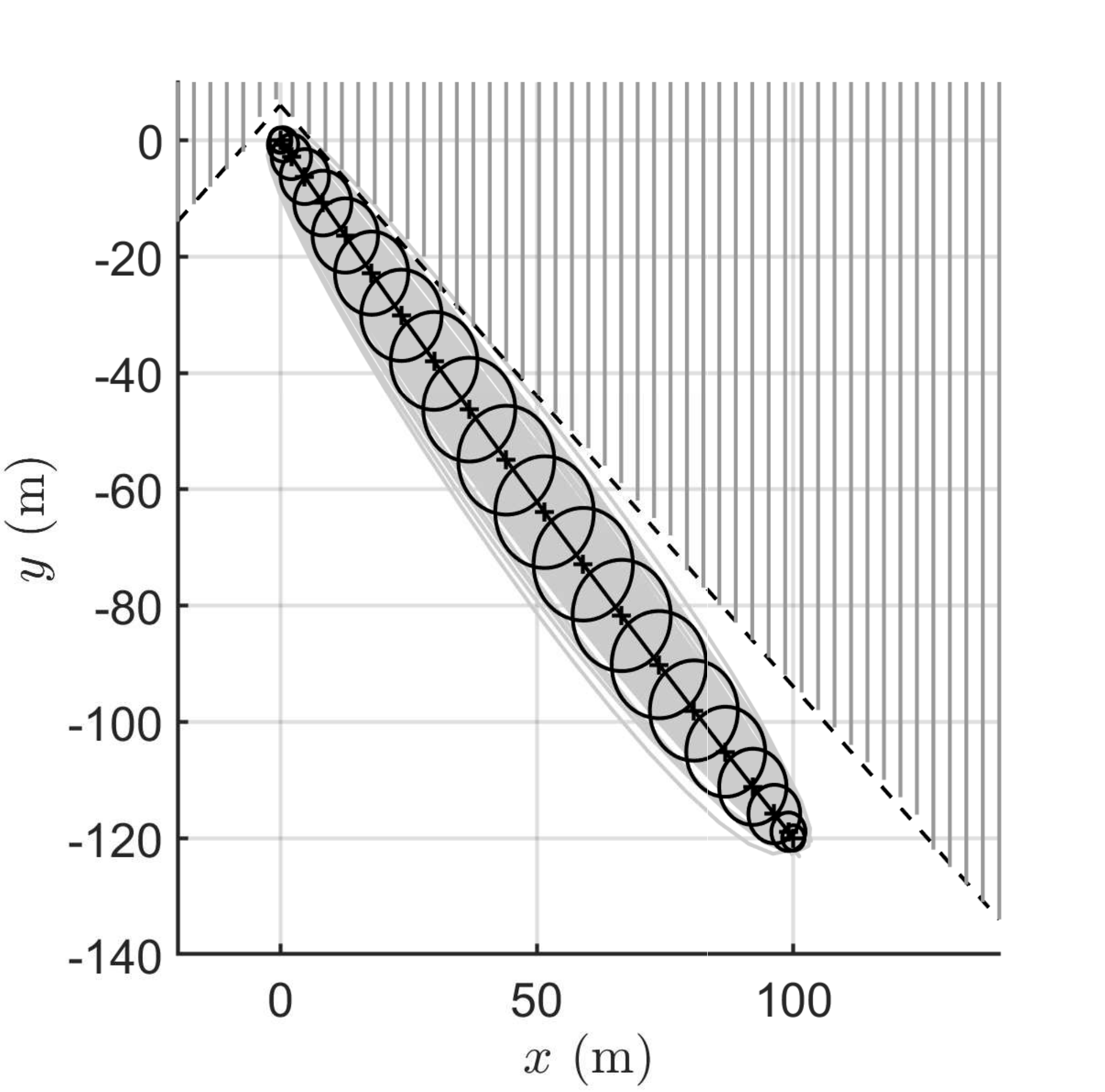}
    \caption{Gaussian chance constraint solution with optimal risk allocation.}
    \label{fig:Gaussian_SC_optimal}
\end{subfigure}
\caption{Trajectories of 500 independent Monte Carlo samples for distributionally-robust and Gaussian polytopic chance constraints with $3\sigma$ covariance ellipses. Each subplot displays the projection of the full state onto the $x-y$ plane.}
\label{fig:SC_trajs}
\end{figure*}

The Monte Carlo results of DR covariance steering are shown in Figure~\ref{fig:SC_trajs}, which show the $x-y$ projection of the state trajectories and 3$\sigma$ covariance ellipses for three different risk solutions.
Notice that the Gaussian solution in Figure~\ref{fig:Gaussian_SC_optimal} remains too close to the boundary of the state space, which implies that under non-Gaussian noises, a CS controller based on Gaussian risk constraints leads to a significant miscalculation of risk.
On the other hand, the DR solution in Figure~\ref{fig:DR_SC_uniform} steers to the middle of the space to ensure proper constraint satisfaction; but this too is under suboptimal risk placement, from which Figure~\ref{fig:DR_SC_optimal} arises as the optimal trajectories corresponding to the optimal risk budget.
Figure~\ref{fig:optimal_costs} shows the monotonically decreasing costs with each iteration of the IRA scheme, which implies the optimal risk budget has resulted in the lowest possible cost.

\begin{figure}[!htb]
    \centering
    \includegraphics[width=\linewidth]{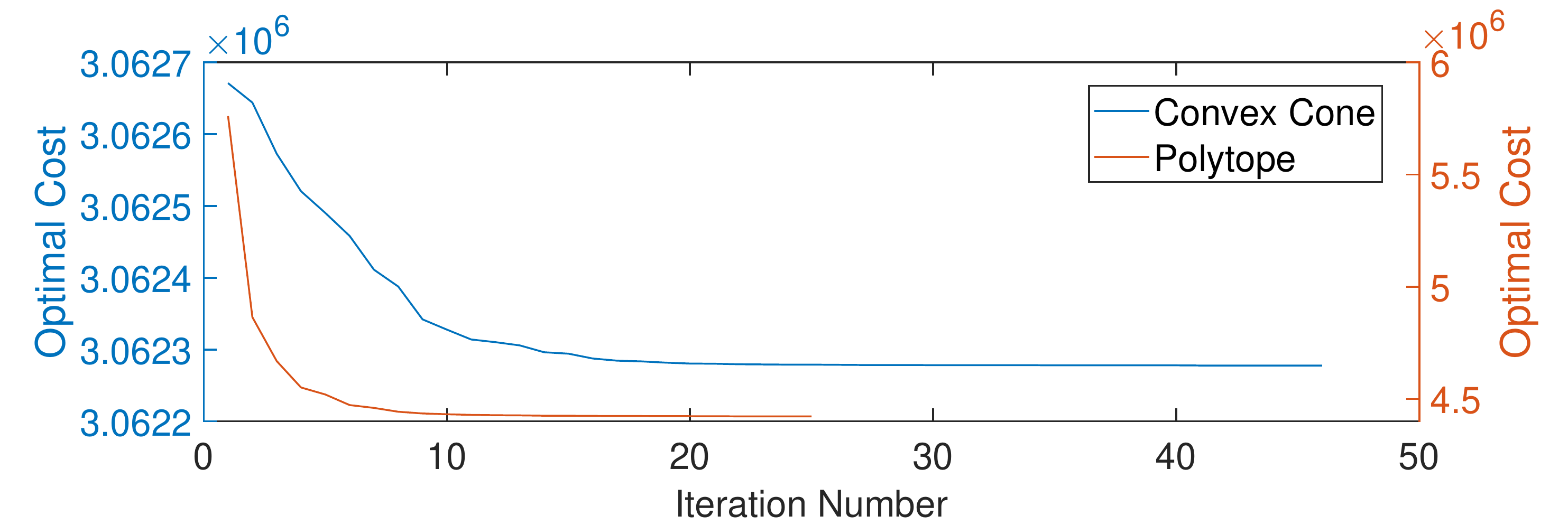}
    \caption{Optimal cost for each DR-IRA algorithm iteration for polytope and cone chance constraints.}
    \label{fig:optimal_costs}
\end{figure}

\begin{figure}[!htb]
    \centering
    \includegraphics[width=\linewidth]{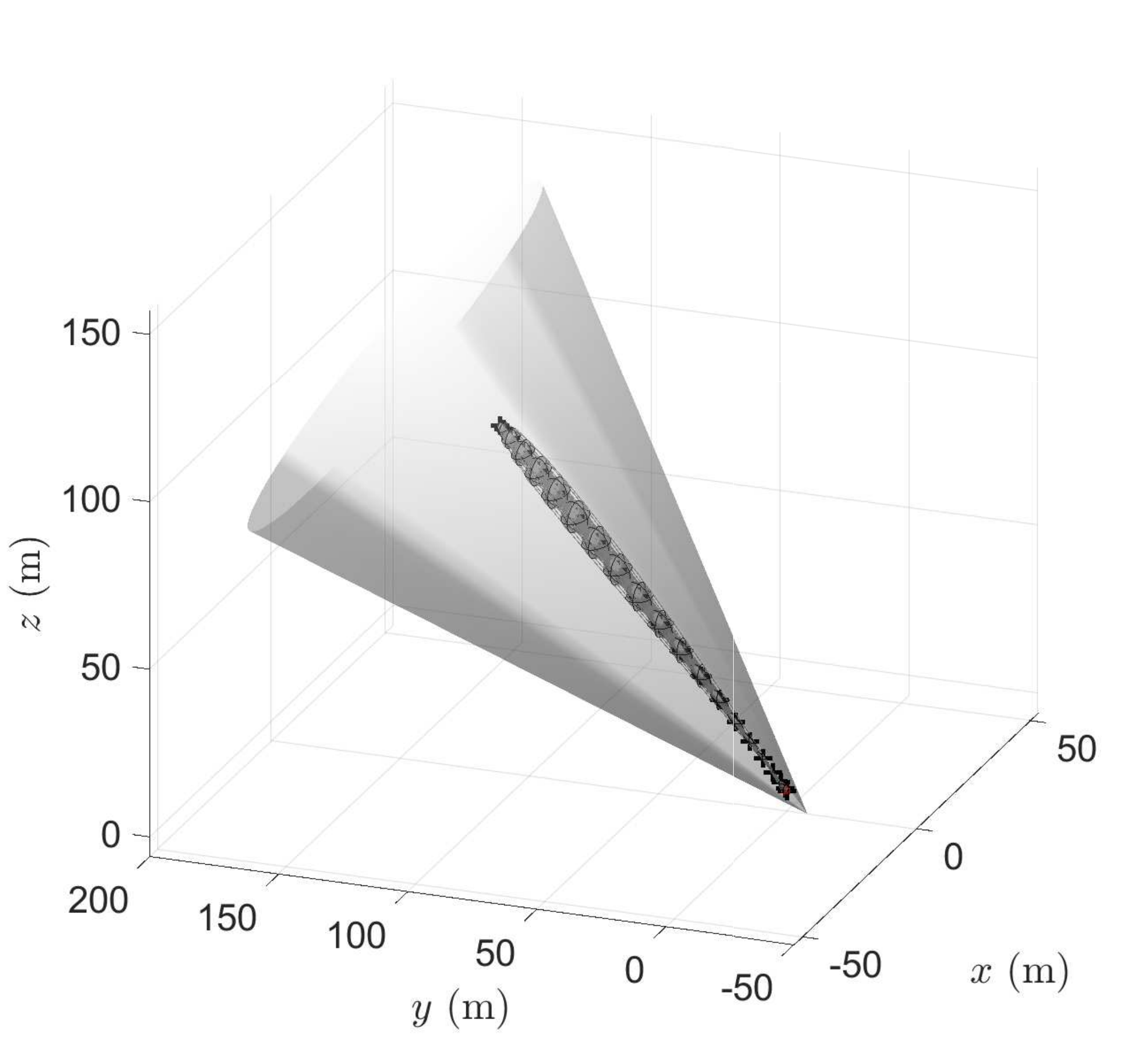}
    \caption{Optimal trajectories for 500 independent Monte Carlo trials under DR cone chance constraints.}
    \label{fig:trajs_cone}
\end{figure}

The simulation results of covariance steering with distributionally robust iterative risk allocation for convex conic state risk constraints using reverse union bound approximation are shown in Figure~\ref{fig:trajs_cone}, with the cost versus the iteration trade-off shown in Figure~\ref{fig:optimal_costs}.
The state trajectories and their 3$\sigma$ dispersions remain well within the cone at all time steps and are robust to any zero mean, unit covariance disturbance.

\section{Conclusion and Future Directions}
\label{sec_conclusion}

In this article we have incorporated an DR-IRA strategy to optimize the worst case probability of violating the state constraints at every time step within the CS problem of a linear stochastic system subject to distributionally robust risk constraints. 
The use of DR-IRA in the context of CS with distributionally robust risk constraints results in optimal solutions that have a true risk much closer to the intended design requirements, compared to the use of a uniform risk allocation. 
We also extended the approach to quadratic chance constraints in the form of convex cones.
Future work will seek to solve the CS problem with output feedback and also extend the problem setting to handle nonlinear dynamics. Further, the general problem of moment steering involving the first $k \in \mathcal{N}$ moments is interesting to be investigated as well.

\noindent \textbf{Acknowledgment:}
The work of the first author has been supported by the European Research Council (ERC) under the European Union’s Horizon 2020 research and innovation program under grant agreement No 834142 (Scalable Control). The work of the second and third author has been supported by NASA University Leadership Initiative award 80NSSC20M0163 and ONR award N00014-18-1-2828.



\bibliographystyle{IEEEtran}
\bibliography{bibliograph}

\section*{Appendix}

\section*{A.~Proof of Lemma 1}
\setcounter{equation}{0}
\renewcommand{\theequation}{A.\arabic{equation}}

We drop the subscript $k$ denoting the time index for convenience in the following proof. Let the original DR conic risk constraint be given by
\begin{align*} 
    \sup_{\bbp_{x} \in \mathcal{P}^{x}} \, \bbp_{x} \left[ \underbrace{\norm{Ax+b}_{2}}_{:= \xi} \leq \underbrace{c^{\top} x + d}_{:= \eta} \right] \geq 1 - \delta.
\end{align*} 
Note that the state $x$ is a random vector modeled using moment based ambiguity set. That is, $x \sim \mathbb{P}_{x} \in \mathcal{P}^{x}$ where,
\begin{align*}
    \mathcal{P}^{x} := \left \{ \mathbb{P}_{x} \mid \mathbb{E}[x] = \mu, \mathbb{E}[(x-\mu)(x-\mu)^{\top}] = \Sigma \right \}.
\end{align*}
Note that $\xi$ follows some unknown distribution with probability density function $f_{\xi}(x)$. 
On the other hand, $\eta$ can be modeled using a moment-based ambiguity set similar to that of $x$. 
That is, $\eta \sim \mathbb{P}_{\eta} \in \mathcal{P}^{\eta}$ where,
\begin{equation*} \label{eqn_eta_ambig_set}
    \mathcal{P}^{\eta} := \left \{ \mathbb{P}_{\eta} \mid \mathbb{E}[\eta] = \underbrace{c^{\top}\mu + d}_{:= \Bar{\eta}}, \mathbb{E}[(\eta-\Bar{\eta})(\eta-\Bar{\eta})^{\top}] = c^{\top} \Sigma c \right \}.
\end{equation*}
Let us define a standard scalar random variable $z$ using moment based ambiguity set. That is, $z \sim \mathbb{P}_{z} \in \mathcal{P}^{z}$ where,
\begin{align*}
    \mathcal{P}^{z} := \left \{ \mathbb{P}_{z} \mid \mathbb{E}[z] = 0, \mathbb{E}[zz^{\top}] = 1 \right \}.
\end{align*}
Then, using $z$ we can rewrite the random variable $\eta$ as follows,
\begin{align*}
    \eta = \Bar{\eta} + z \sqrt{c^{\top} \Sigma c}.
\end{align*}
Let $f_{\xi, \eta}(\cdot, \cdot)$ denote the joint probability density function of the random variables $\xi, \eta$. Then,
\begin{align} \label{eqn_joint_pdf}
    \sup_{\bbp_{x} \in \mathcal{P}^{x}} \, \bbp_{x} \left[ \xi \leq \eta \right] &= \int^{\infty}_{-\infty} \int^{y}_{-\infty} f_{\xi, \eta}(x, y) \, dx \, dy.
\end{align}
The innermost integral of \eqref{eqn_joint_pdf} can be written as
\begin{align*}
    \int^{y}_{-\infty} f_{\xi, \eta}(x, y) \, dx &= \int^{\Bar{\eta} + y \sqrt{c^{\top} \Sigma c}}_{-\infty} f_{\xi, z}(x, y) \, dx. \\
    \implies \sup_{\bbp_{x} \in \mathcal{P}^{x}} \, \bbp_{x} \left[ \xi \leq \eta \right] &= \int^{\infty}_{-\infty} \int^{y}_{-\infty} f_{\xi, \eta}(x, y) \, dx \, dy \\
    &= \int^{\infty}_{-\infty} \int^{\Bar{\eta} + y \sqrt{c^{\top} \Sigma c}}_{-\infty} f_{\xi, z}(x, y) \, dx \, dy \\
    &\geq \int^{\infty}_{-\infty} \int^{\Bar{\eta}}_{-\infty} f_{\xi, z}(x, y) \, dx \, dy \\
    &= \int^{\Bar{\eta}}_{-\infty} \underbrace{\int^{\infty}_{-\infty} f_{\xi, z}(x, y) \, dy}_{:= f_{\xi}(x)} \, dx \\
    &=: \sup_{\bbp_{x} \in \mathcal{P}^{x}} \, \bbp_{x} \left[ \xi \leq \Bar{\eta} \right]
\end{align*}
Hence, the above DR quadratic risk constraint is a relaxation of the original DR conic risk constraint. 
It follows that if, for all $\delta \in (0,0.5]$,
\begin{align*}
    \sup_{\bbp_{x} \in \mathcal{P}^{x}} \, \bbp_{x} \left[ \xi \leq \Bar{\eta} \right] \geq 1 - \delta,
    \end{align*}
    then
\begin{align*}
    \sup_{\bbp_{x} \in \mathcal{P}^{x}} \, \bbp_{x} \left[ \xi \leq \eta \right] \geq \sup_{\bbp_{x} \in \mathcal{P}^{x}} \, \bbp_{x} \left[ \xi \leq \Bar{\eta} \right] \geq 1 - \delta,
\end{align*}
and the result follows.


\section*{B.~Proof of Proposition 2}
\setcounter{equation}{0}
\renewcommand{\theequation}{B.\arabic{equation}}
For each $i = 1,\dots,n$ at time $k$, we rewrite \eqref{eqn_main_dr_quad_constraint} as
\begin{align} \label{eqn_prop2_proof_1}
    \sup_{\bbp_{x_{k}} \in \mathcal{P}^{x_{k}}} \, \bbp_{x_{k}} \left[ -f_{i,k} - b_{i} \leq a^{\top}_{i} x_{k} \leq f_{i,k} - b_{i} \right] &\geq 1 - \beta_i \delta_{k}.
\end{align}
Define the set
\begin{align}
\Omega_{k} := \left\{ x_{k} \mid a^{\top}_{i} x_{k} \leq f_{i,k} - b_{i} \wedge a^{\top}_{i} x_{k} \geq -f_{i,k} - b_{i} \right\}.    
\end{align}
Then, \eqref{eqn_prop2_proof_1} can be equivalently written as
\begin{align}
    \sup_{\bbp_{x_{k}} \in \mathcal{P}^{x_{k}}} \, \bbp_{x_{k}} \left[ x_{k} \in \Omega_{k} \right] &\geq 1 - \beta_i \delta_{k}.
\end{align}
Applying the reverse union bound \cite{hunter_book} to events in $\Omega_{k}$, we get
\begin{align}
    \sup_{\bbp_{x} \in \mathcal{P}^{x}} \, \bbp_{x} \left[ x_{k} \in \Omega_{k} \right] &\geq \sup_{\bbp_{x} \in \mathcal{P}^{x}} \, \bbp_{x} \left[ a^{\top}_{i} x_{k} \leq f_{i,k} - b_{i} \right] \nonumber \\ 
    &+ \sup_{\bbp_{x} \in \mathcal{P}^{x}} \, \bbp_{x} \left[ a^{\top}_{i} x_{k} \geq -f_{i,k} - b_{i} \right] - 1.
\end{align}
Hence, if the constraint
\begin{align} \label{eqn_prop2_proof_2}
    &\sup_{\bbp_{x} \in \mathcal{P}^{x}} \, \bbp_{x} \left[ a^{\top}_{i} x_{k} \leq f_{i,k} - b_{i} \right] \nonumber \\
    &+ \sup_{\bbp_{x} \in \mathcal{P}^{x}} \, \bbp_{x} \left[ a^{\top}_{i} x_{k} \geq -f_{i,k} - b_{i} \right] \geq 2 - \beta_i \delta_{k}
\end{align}
holds, then \eqref{eqn_prop2_proof_1} is satisfied and hence \eqref{eqn_main_dr_quad_constraint} is also satisfied.
Note that \eqref{eqn_prop2_proof_2} is equivalent to the following decoupled constraints
\begin{subequations}
\begin{align} 
    \sup_{\bbp_{x} \in \mathcal{P}^{x}} \, \bbp_{x} \left[ a^{\top}_{i} x_{k} \leq f_{i,k} - b_{i} \right] &\geq \epsilon^{1}_{i,k}, \label{eqn_prop2_proof_3} \\
    \sup_{\bbp_{x} \in \mathcal{P}^{x}} \, \bbp_{x} \left[ a^{\top}_{i} x_{k} \geq -f_{i,k} - b_{i} \right] &\geq \epsilon^{2}_{i,k}, \label{eqn_prop2_proof_4} \\ 
    \epsilon^{1}_{i,k} + \epsilon^{2}_{i,k} &\geq 2 - \beta_i \delta_{k}.
\end{align}
\end{subequations}
Note that $z_{k} := a^{\top}_{i} x_{k}$ is a random variable with mean $\bar{z}_{k} := a^{\top}_{i} \bar{x}_{k}$ and covariance $\Sigma_{z_{k}} := a^{\top}_{i} E_{k} \Sigma_{X} E^{\top}_{k} a_{i}$ and the unknown true distribution of $z_{k}$, namely $\bbp_{z_{k}}$ lies in the moment based ambiguity set $\mathcal{P}^{z_{k}}$ given by
\begin{align*}
    \mathcal{P}^{z_{k}} := \left \{ \mathbb{P}_{z_{k}} \mid \mathbb{E}[z_{k}] = \bar{z}_{k}, \mathbb{E}[(z_{k}-\bar{z}_{k})(z_{k}-\bar{z}_{k})^{\top}] = \Sigma_{z_{k}} \right \}.
\end{align*}
Using Cantelli's inequality on the set of distributions defined using $\mathcal{P}^{z_{k}}$, the required probabilities given by \eqref{eqn_prop2_proof_3} and \eqref{eqn_prop2_proof_4} can be expressed in terms of the DR quantile function $\mathcal{Q}(\cdot)$ as follows
\begin{align}
    \frac{f_{i,k} - b_{i} - a^{\top}_{i} \bar{x}_{k}}{\sqrt{\Sigma_{z_{k}}}} \geq \mathcal{Q}(\epsilon^{1}_{i,k}), \\
    \frac{f_{i,k} + b_{i} + a^{\top}_{i} \bar{x}_{k}}{\sqrt{\Sigma_{z_{k}}}} \geq \mathcal{Q}(\epsilon^{2}_{i,k}).
\end{align}
Substituting the mean and the covariance dynamics given by \eqref{eqn_mean_prop} and \eqref{eqn_cov_prop} respectively, the desired result in obtained. 

\end{document}